\documentclass[a4paper,10pt]{article}

\usepackage{a4wide,amsthm,amsmath, amsfonts}
\usepackage{hyperref}
\usepackage{graphicx}
\usepackage{natbib}

\newtheorem{theorem}{Theorem}[section]

\newtheorem{lemma}{Lemma}[section]
\newtheorem{corollary}{Corollary}[section]

\theoremstyle{definition}

\newtheorem{example}{Example}[section]

\begin{document}

\title{The Discrepancy Principle for Choosing Bandwidths in Kernel Density Estimation}
\author{Thoralf Mildenberger\thanks{This work has been supported
by the Collaborative Research Center {\em Statistical modelling of nonlinear
dynamic processes} (SFB 823, Project C1) of the German Research Foundation (DFG).}\\
Department of Mathematics, Physics, and Computer Science \\
University of Bayreuth \\ thoralf.mildenberger@uni-bayreuth.de}
\date{\today}

\maketitle

\begin{abstract}
We investigate the discrepancy principle for choosing smoothing parameters 
for kernel density estimation.  The method is based on
 the distance between the empirical and estimated distribution functions.
We prove some new positive and negative results on $L_1$-consistency of kernel estimators with bandwidths
chosen using the discrepancy principle. Consistency crucially depends on
a rather weak H\"older condition on the distribution function. 
We also unify and extend previous results on the behavior of the chosen bandwidth under
more strict smoothness assumptions. Furthermore, we
compare the discrepancy principle to standard methods in a simulation study. 
Surprisingly, some of the proposals work reasonably well over a large set of different densities and sample sizes, and
the performance of the methods at least up to $n=2500$ can be quite different from their asymptotic behavior.   
\end{abstract}

\section{Introduction}

We investigate the {\em discrepancy principle}, a simple method for choosing 
the bandwidth in kernel density 
estimation which -- unlike most other methods like cross-validation or plug-in estimates -- does not 
directly aim at minimizing the risk. 

In the following, let $X_1,\dots,X_n$ denote iid random variables having a distribution 
with Lebesgue density $f$ and distribution function $F$. We denote the empirical distribution function 
by $F_n$. 

A function $K:\mathbb{R}\longrightarrow \mathbb{R}$ is called a {\em kernel} of order $\ell$ for 
 $\ell \in \mathbb{N}$, if $u^jK(u) \in L_1(\mathbb{R})$ for $j=0,\dots,\ell$ and
\begin{align*}
\int u^j K(u)du & = \begin{cases}
                   1 & (j=0) \cr
		  0 & (j=1,\dots,\ell-1) \cr
		  k_\ell \in \mathbb{R}\setminus\{0\} & (j=\ell)
\end{cases}.
\end{align*}

For a kernel $K$ and $h>0$ we define $K_h(u):= h^{-1}K(h^{-1}u)$. 
We denote the distribution function associated with $K$ 
(which is not necessarily monotone if $K$ is not a probability density) by $\mathbb{K}$.
For iid random variables $X_1,\dots,X_n$, a kernel $K$ of order $\ell \in \mathbb{N}$ 
and a {\em bandwidth} $h>0$ the function $x \longrightarrow \hat f_h(x)$ given by
\begin{align*}
\hat f_h (x) := \frac{1}{nh} \sum_{i=1}^n K \left(\frac{x-X_i}{h} \right) = \frac{1}{n} \sum_{i=1}^n K_h (x-X_i)
\end{align*}
is called the {\em kernel density estimator}. A corresponding {\em kernel estimator of the distribution function} 
is given by
\begin{align}
\hat F_n^h (x) := \int_{-\infty}^x \hat f_h (t) dt = \frac{1}{n} \sum_{i=1}^n
\mathbb{K} \left(\frac{x-X_i}{h} \right) = (F_n \ast K_h)(x). \label{kerdistf}
\end{align}

More important than the choice of the kernel is the choice of the bandwidth $h$. Depending on
the risk function and on additional assumptions on $f$, often an explicit expression for the 
(at least asymptotically) optimal value can be derived. However, it necessarily depends
on some functionals of the unknown true density $f$. 
Most parameter choice strategies used in practice aim at minimizing the risk. In contrast, the strategies 
considered here are based on a measure of distance between the empirical and estimated distribution
functions, i.e. a direct comparison of the estimate with the data. 

In the following, by the 
{\em discrepancy principle for choosing the bandwidth for kernel density estimators} we mean that 
 $h$ is chosen such that
\begin{align}
d(F_n, \hat F_n^h) = s(n). \label{dpgleich}
\end{align}
The {\em threshold function} $s:\mathbb{N}\longrightarrow \mathbb{R}^+$ depends on $n$ only and fulfills
$s(n) = o(1)$ for $n \rightarrow \infty$. For the distance $d$ between distribution functions we will
always take the Kolmogorov or (generalized) Kuiper distances although, in principle,  other metrics could
be used. The different suggestions in the previous literature differ in their choices of $s(n)$ and $d$
and possibly in their prescriptions for the selection of a solution of (\ref{dpgleich})
in case there are multiple solutions.   

The discrepancy principle was first introduced by \citet{Mor66solu} in the context of
 (deterministic) inverse problem theory, where it is one of the most 
widely known methods for choosing a regularization parameter. In Statistical Learning Theory, the connection 
between nonparametric statistics and ill-posed problems is strongly emphasized, 
and already in the seventies density estimation was recognized as being closely related to the 
problem of numerical differentiation, which is an ill-posed problem. Methods adapted from
deterministic inverse problem theory as well as using the discrepancy principle for choosing their
smoothing parameters have been suggested by \cite{VapSte78density} and \cite{AidVap89dens}, 
see also Chapter 7 in \cite{VapLearning} and Chapter 7 in \cite{VapNature} for 
detailed accounts. 

Variants of the discrepancy principle (but under different names) have also independently 
been proposed in the context of the so-called {\em Data Features} or {\em Data Approximation} 
approach \citep{Dav95data, Dav08data} which has its roots in robust statistics and exploratory data analysis. 
The main idea is to choose the simplest  estimate (with simplicity e.g. measured by smoothness) 
that is sufficiently close to the data.
Several procedures for density estimation based on these ideas have been proposed, including methods based on 
kernel density estimators \citep{Dav95data}, regular histograms \citep{DavGatNor09histogram} and 
the taut-string estimator \citep{DavKov04dens}.

The discrepancy principle has also been used in a few other approaches to density
 estimation.  \cite{ELRdisc} suggest a version for kernel density estimation 
that chooses a bandwidth of the optimal order  
under standard assumptions; see also \citet[Ch. 7.6]{EggLaR1}. 
The same authors also use their method for
choosing a penalty parameter in a penalized-likelihood approach \citep[Ch. 7.7]{EggLaR1} and in a density deconvolution 
method \citep{ELRdeconv}. 

The different variants of the discrepancy principle for density estimation mentioned above have largely been suggested
 independently of each other,
and to our knowledge, there has never been a systematic investigation of this approach.  

In Section 2, we show that a solution of (\ref{dpgleich}) exists under very weak conditions, and we
show that the almost sure $L_1$-consistency of the resulting kernel density estimate mainly depends on a rather mild
H\"older condition on the distribution function $F$. This condition is, for example, fulfilled for all 
square-integrable densities provided that the threshold function $s$ decays slowly enough. We also 
give sufficient conditions for the resulting estimator to be inconsistent. 
In Section 3 we extend and unify some known results on the exact order of the chosen bandwidth.
 Furthermore, we compare  
different versions of the discrepancy principle with standard methods of smoothing parameter
selection in a simulation study (Section 4). 
The methods can behave quite differently to what is predicted by the asymptotic results 
even for sample sizes up to at least $n=2500$. This is not so much of a surprise as the asymptotics are
mostly based on the law of the iterated logarithm for the empirical distribution function.
 Indeed, some versions of the discrepancy principle that were previously 
suggested in the literature perform reasonably well over a wide range of different densities, 
while others suffer from oversmoothing for these sample sizes, although they are
guaranteed to undersmooth asymptotically. 
The last section contains some concluding remarks.

\section{Existence and consistency}

First, we investigate the existence of a solution of (\ref{dpgleich}). We measure the distance between two distribution functions  
$F$ and $G$ either by the {\em Kolmogorov distance}
$$ d_\infty(F,G):= \| F - G \|_\infty $$ or by the {\em $k$-th order Kuiper distance} (for $k \in \mathbb{N}$) first introduced in \cite{DavKov04dens} and defined by
$$ d_{kuip,k}(F,G) := \sup_{a_1 \leq b_1 \leq a_2 \leq b_2 \leq \dots \leq a_k
\leq b_k}
\sum_{i=1}^k | (F(b_i) - F(a_i))-(G(b_i) - G(a_i))| .$$
For a continuous probability distribution function $F$ and the empirical distribution $F_n$ of a sample of size $n$ drawn from $F$, the distributions of 
$d_\infty(F_n,F)$ and $d_{kuip,k}(F_n,F)$ do not depend on $F$. 
For $k=1$ we obtain the usual Kuiper distance.
All these distances are topologically equivalent and it is easy to see that
\begin{align*}
d_\infty(F,G) \leq d_{kuip,k}(F,G) \leq 2k d_\infty(F,G).
\end{align*}
In the following, we always have $d=d_\infty$ or $d=d_{kuip,k}$ for some $k \in \mathbb{N}$, and we define 
\begin{align*}
c_d=\begin{cases} 1, & d=d_\infty \cr 2k, &
d=d_{kuip,k}
\end{cases}.
\end{align*}
It should be noted that, since we allow for higher order kernels, some distribution
functions do not correspond to probability measures but to signed measures.

For a kernel $K$ with associated distribution function $\mathbb{K}$, we define
$$\kappa_0 := \sup_{x \in \mathbb{R}} | \mathbb{K}(x) - F_0(x)|,  $$
where $F_0(x):=\mathbb{I}(x \geq 0)$ is the distribution function of the Dirac measure in $0$.
In case $K$ is a probability density, we have $\kappa_0=\max
\{\mathbb{K}(0),1-\mathbb{K}(0)\}$. If $K$ is also symmetric around zero, he have $\kappa_0=\mathbb{K}(0)=1/2$. 

\begin{figure}[t]
\begin{center}
\includegraphics[width=11cm,height=9cm]{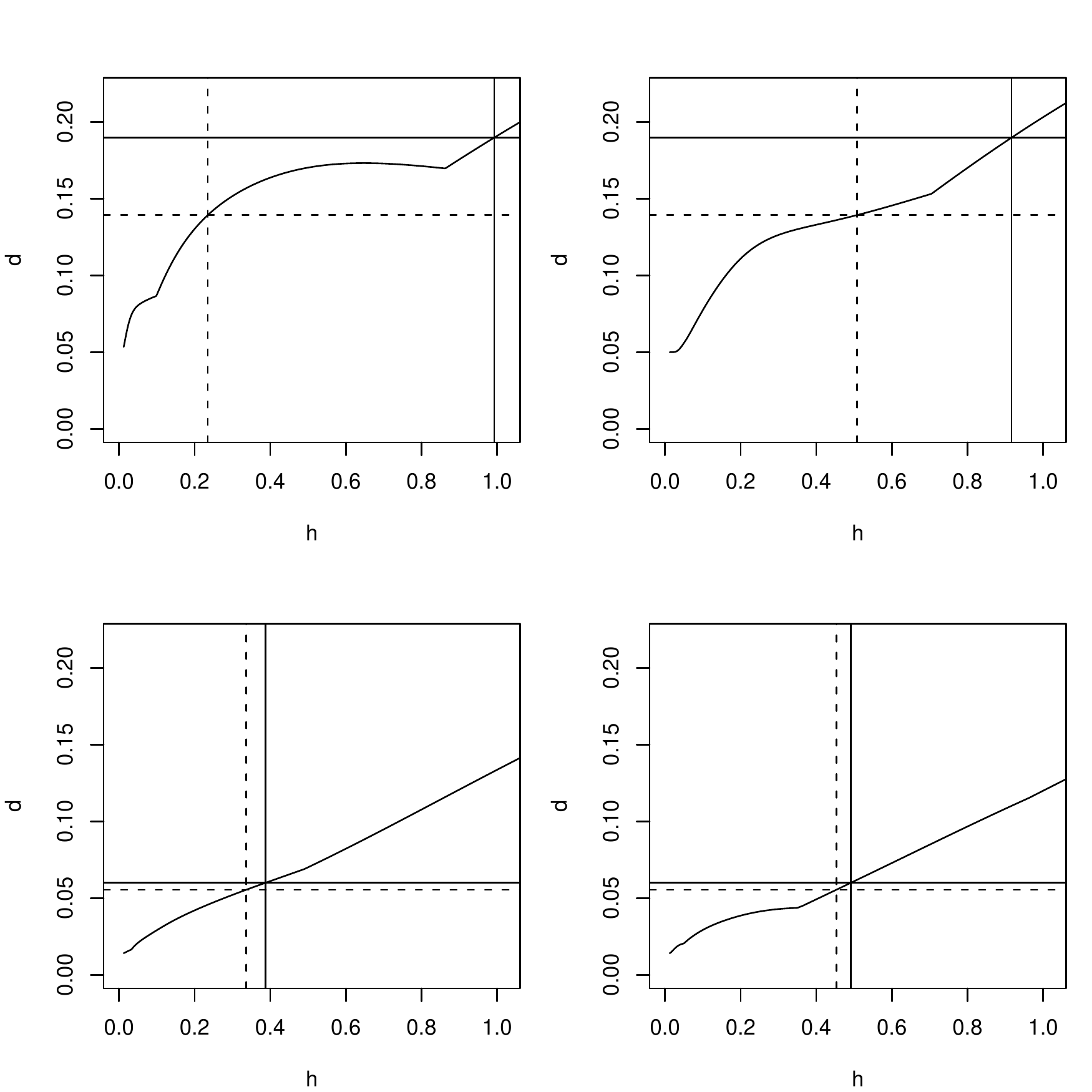}
\caption{Solutions of $d_\infty(\hat F_n^h, F_n)=s(n)$ using
a standard Gaussian kernel for $X_1,\dots,X_n \sim N(0,1)$. Top row: $n=10$, bottom row: $n=100$.
Straight lines: $s(n)=0.6n^{-1/2}$, broken lines:
$s(n)=0.35n^{-2/5}$.}\label{kerdpbild}
\end{center}
\end{figure}

The following lemma shows that, almost surely, for 
fixed $n$, the function $h \longrightarrow d(F_n,\hat F_n^h)$ is continuous and must -- under weak conditions on $s$ -- 
take the value $s(n)$ for at least one $h$ if $n$ is large enough. An analogous statement has been proved by  \cite{ELRdisc,EggLaR1} for the special case of a 
symmetric, nonnegative kernel of of order $2$ and $d=d_\infty$. The proof can be found in \citet{Diss}, pp. 27-28.

\begin{lemma}\label{hexist}
For $F_n$ an empirical distribution function of an iid sample from a distribution
with continuous distribution function and $\hat F_n^h$ as in 
(\ref{kerdistf}) we have almost surely:
\begin{enumerate}
\item $d(F_n,\hat F_n^h)$ is continuous in $h$.
\item $\liminf_{h \rightarrow 0} d(F_n,\hat F_n^h) \leq c_d \frac{\kappa_0}{n}$.
\item $\limsup_{h \rightarrow \infty} d(F_n,\hat F_n^h) \geq \kappa_0$.
\end{enumerate}
\end{lemma}

Lemma \ref{hexist} shows that if $s(n)=o(1)$ and $n^{-1}=o(s(n))$ the equation $d(F_n,\hat F_n^h) =
s(n)$ almost surely has at least one solution $h_{s,n}$ for sufficiently large $n$. 
These conditions are fulfilled by the threshold functions previously 
proposed in the literature. Moreover, the minimum sample size that guarantees 
existence of at least one solution can be calculated explicitly since it 
depends on $s(n)$ and $\kappa_0$
only, and not on the sample or on the underlying true distribution (assuming there are no ties, which holds true almost surely). 
For example, if $s(n)=0.6n^{-1/2}$ as proposed by \citet[Ch. 7.9]{VapLearning}
or $s(n)=0.35n^{-2/5}$ as proposed in \citet{ELRdisc}, $d=d_\infty$ and $K$ is any symmetric probability density, we have that $s(n) \in [\frac{1}{2n},\frac{1}{2}]$ for  
$n\geq 2$, so existence of the bandwidth can be guaranteed if there are at least two data points. As already noted by \citet{ELRdisc}, the
function $h \longrightarrow d(F_n,\hat F_n^h)$ is not necessarily monotone, 
so that the bandwidth chosen according to the discrepancy principle is 
not necessarily unique; \cite{ELRdisc} suggest using the smallest solution while the Data Approximation approach would suggest using the largest one.
However, none of the results given subsequently depends on the particular choice of the solution, and multiple solutions seem to 
occur only rarely in larger samples. 

Figure \ref{kerdpbild} shows shows two realizations each for $n=10$ and $n=100$.
 The samples were drawn from a standard normal distribution 
and the Gaussian kernel was used. 
The horizontal lines correspond to the two different choices of the threshold 
functions mentioned above. The solution $d(F_n,\hat F_n^h) = s(n)$ can be 
computed numerically since the function $h \rightarrow d(F_n,\hat F_n^h)$ is continuous.
In \cite{ELRdisc}, a secant method is proposed for solving this equation, but we use the
related regula falsi which we found to be more stable. The possibility of using 
an iterative method makes selection of the bandwidth using the discrepancy quite fast
in comparison to other methods (like cross-validation) where one usually has to evaluate
some criterion on a grid of possible bandwidths. In addition, well-known 
formulas exist for calculating Kolmogorov- and Kuiper-distances (for $k=1$) for two
distribution functions and these can be applied if $K$ is a probability density.

In the following, we frequently need the function 
$$F_h:=F\ast K_h .$$ In case 
$K_h$ is a probability density, $F_h$ is a probability distribution function, otherwise it is the distribution function of a signed measure.

The proof of the following Lemma is based on basic properties of convolutions and the Law of the Iterated Logarithm, see
\cite{Diss}, p. 29, for details.

\begin{lemma} With probability $1$, 
\begin{enumerate}
\item $d(F_n,F)=O\left((\log\log n/n)^{1/2}\right)$ and
\item $d(\hat F^h_n,F_h) =O\left((\log\log n/n)^{1/2}\right)$ uniformly in $h$.
\end{enumerate}
\end{lemma}

The next theorem shows that bandwidths chosen using the discrepancy principle converge to $0$ almost surely. 
This result will be needed later on for obtaining 
more precise statements about the behavior of the selected bandwidths. At this point, 
$F$ must be continuous but does
not need to have a density. As a by-product, the theorem also shows that the resulting
estimator for the distribution function is always consistent w.r.t. $d$, although our aim is to estimate 
the density rather than the distribution function.
The proof of the second assertion is based on similar Fourier arguments as the proof of Theorem 3 in \cite{Yam73uniform}.

\begin{theorem}\label{hgegen0} Let $F$ be a continuous distribution function, $F_n$
and $\hat F_n^h$ as above and $s(n)=o(1)$. For the bandwidth $h_{s,n}$ chosen as a solution of$$
d(F_n,\hat F_n^h) = s(n). $$ we have almost surely
\begin{enumerate}
 \item $d(F,\hat F_n^{h_{s,n}}) \longrightarrow 0$ and
 \item $h_{s,n} \longrightarrow 0.$
\end{enumerate}
\end{theorem}
\begin{proof}
1. With probability 1, we have: \begin{align*}
d(F,F_{h_{s,n}}) &\leq d(F, F_n) + d(F_n, \hat F_n^{h_{s,n}}) + d(\hat
F_n^{h_{s,n}}, F_{h_{s,n}}) \\
&= O\left((\log\log n/n)^{1/2}\right) + s(n) + O\left((\log\log n/n)^{1/2}\right)\\
 &= o(1), 
\end{align*}
and hence $$ d(F,\hat F_n^{h_{s,n}}) \leq d(F,F_{h_{s,n}}) + d(F_{h_{s,n}},\hat F_n^{h_{s,n}})
= o(1).$$
2. According to the first part, $d_\infty(F, F_{h_{s,n}}) \leq d(F,
F_{h_{s,n}}) \longrightarrow 0$ with probability $1$; it remains to show that this implies $h_{s,n}
\stackrel{a.s.}{\longrightarrow} 0$. In the following $h_n:=h_{s,n}$ denotes the sequence of bandwidths chosen, $P$ denotes
the probability measure associated with $F$ and 
$\mu_{h_n}$ the (signed) measure with Lebesgue density $K_{h_n}$. 
Denote by $\hat P$, $\hat K$ and $\hat K_{h_n}$ the Fourier transforms of $P$, $K$ and $K_{h_n}$, 
respectively. Observing that the sequence $(|P \ast \mu_{h_n}|)_{n \in \mathbb{N}}$ is tight \citep[pp. 30-31]{Diss}
and combining Proposition 8.1.8 in \cite{Bogachev2} with a result on page 173/174 in \cite{Katznelson}, it follows that 
 $\hat P \hat K_{h_n} (t) \longrightarrow \hat P(t)$ for all $t \in \mathbb{R}$. 
Because of the continuity of the Fourier transform, we must have $\hat P>0$ on an interval $[-\varepsilon,\varepsilon]$ for some $\varepsilon > 0$,
 which implies that $\hat K_{h_n}(t) = \hat K(h_nt) \longrightarrow 1$ for all $t \in [-\varepsilon,\varepsilon]$.
Since $\int u^\ell K(u)du \neq 0$, $\hat K$ cannot be identically $1$ on any interval around zero. 
But this implies that $h_n \longrightarrow 0$.
\end{proof}

While the previous results hold for any continuous distribution function $F$, for the remainder of the paper
we suppose that a Lebesgue density $f$ exists. 

For consistency of the kernel density estimate with bandwidth chosen by the discrepancy principle, 
we also need that the chosen bandwidth does not go to 
$0$ too quickly. This can be guaranteed under rather mild conditions. For $0 < \alpha \leq 1$, let 
$$C^{0,\alpha}:=\left\{ F \left| F:\mathbb{R} \longrightarrow \mathbb{R} \text{ and } \exists C>0 \text{ with }
\sup_{x,y \in \mathbb{R}} {| F(x) - F(y) |}/{|x-y|^\alpha} \leq C  
 \right. \right\}$$ 
denote the set of all H\"older continuous
functions with exponent $\alpha$. Smoothness of the distribution function follows 
from integrability assumptions on the density. We define 
for $p \in [1,\infty)$ 
$$L_p(\mathbb{R}):=\left\{ f \left| f:\mathbb{R} \longrightarrow \mathbb{R} \text{ and } 
\| f \|_p := \left(\int |f|^p d\lambda \right)^{1/p} < \infty \right. \right\}$$
where $\lambda$ denotes the Lebesgue-Measure on (the Borel sets of) $\mathbb{R}$. Then we have:
\begin{lemma}
Let $f$ denote a probability density and $F$ the corresponding distribution function. 
For $p \in (1,\infty)$ we have 
$$ f \in L_p (\mathbb{R}) \Longrightarrow  F \in C^{0,(p-1)/p}.$$
\end{lemma}
\begin{proof}
For any $x<y \in \mathbb{R}$ and $q=\frac{p}{p-1}$ we obtain using 
the H\"older inequality 
\begin{align*}
| F(y) - F(x) | 
&\leq \| f \|_p (y-x)^{1/q}  
\end{align*}
and hence $F \in C^{0,(p-1)/p}$. 
\end{proof}

This implies for example that for any square-integrable density (i.e., $f \in L_2(\mathbb{R})$)
$F$ is H\"older-continuous with exponent $\alpha=\frac{1}{2}$. We also observe that,
using a similar argument, for any bounded $f$ the corresponding distribution function $F$ is
H\"older-continuous with exponent $\alpha=1$.  

The next theorem shows that $L_1$ consistency of a kernel density 
estimator with bandwidth chosen by the discrepancy principle can be guaranteed if the distribution function
is H\"older continuous with an sufficiently large exponent  and the threshold function goes to $0$
slowly enough. 

\begin{theorem}\label{konsistenz}
Let $K$ be a kernel of order $\ell$, $\ell \geq 1$, and $f$ a density with associated 
distribution function $F$ such that $F \in C^{0,\alpha}$ for some $0 <\alpha \leq 1$. 
If the threshold function $s(n)$ is such that $\sqrt{\frac{\log\log n}{n}}=o(s(n))$ and
 $n^\alpha s(n) \rightarrow \infty$ for $n \rightarrow \infty$, then with probability $1$
we have that $$ nh_{s,n} \rightarrow \infty.$$
\end{theorem}
\begin{proof}
The H\"older condition $F \in C^{0,\alpha}$ implies that there is a constant $A>0$ such that
$d_\infty(F,F_h) \leq A h^\alpha$, cf. \cite{Shapiro}, Theorem 20. 
With probability 1, we have that 
\begin{align*}
n^\alpha s(n) &= n^\alpha d(F_n, \hat F_n^{h_{s,n}}) \\
&\leq c_d n^\alpha \left( d_\infty(F_n,F) + d_\infty(F, F_{h_{s,n}}) + d_\infty(F_{h_{s,n}}, \hat F_n^{h_{s,n}}) \right) \\
&\leq Ac_d n^\alpha h_{s,n}^\alpha + n^\alpha O((\log\log n/n)^{1/2})
\end{align*}
which implies that  
\begin{align*}
Ac_d n^\alpha h_{s,n}^\alpha \geq n^\alpha s(n) (1+o(1)),
\end{align*}
and hence, since $n^\alpha s(n) \rightarrow \infty$, that $nh_{s,n} \rightarrow \infty$.
\end{proof}

\begin{corollary}\label{konscorr} If $K$ is a probability density and $f$ and $s$ are such that the conditions of
Theorem \ref{konsistenz} are fulfilled, we have $$ \lim_{n \rightarrow \infty} \int | \hat f_{h_{s,n}}(x) - f(x) |dx =0 $$ 
with probability $1$.  
\end{corollary}
\begin{proof}
Under the stated conditions, part two of Theorem \ref{hgegen0} and Theorem \ref{konsistenz} yield that
almost surely $h_{s,n} \longrightarrow 0$ and $nh_{s,n} \longrightarrow \infty$, which by Theorem 1 in
Chapter 6 of \cite{L1} implies $ \lim_{n \rightarrow \infty} \int | \hat f_{h_{s,n}}(x) - f(x) |dx =0 $ almost surely.
\end{proof}

From Corollary \ref{konscorr}, we have that almost sure $L_1$-consistency can
be guaranteed for the threshold function $s(n)=0.35n^{-2/5}$ suggested in \cite{ELRdisc},
$K$ a probability density and $f \in L_2$. 

Although the conditions for consistency are rather weak, the resulting density estimate may
be inconsistent if the distribution function is too rough or the threshold function vanishes too quickly:

\begin{theorem}\label{dpinkonsistenzthm}
Let $K$ be a kernel and $0<\varepsilon <1/2$ such that
$n^\varepsilon s(n)=o(1)$. Let $F_n$ denote the empirical distribution function of an 
iid sample drawn from a distribution with density $f$ and distribution function $F$. 
Suppose there exist constants $c,h_0 >0$ such
$$ d_\infty(F,F_h) \geq ch^\varepsilon $$
for all $0 < h < h_0$. Then, if  
$h_{s,n}$ is a solution of 
 $d(F_n,\hat F_n^h) = s(n)$, we have:
\begin{enumerate}
 \item $nh_{s,n} \longrightarrow 0$ with probability $1$ and
 \item if $K$ is compactly supported and there exist $a,b>0$ such that
$\lambda\{ x : f(x) \geq b \} \geq a$, where $\lambda$ denotes Lebesgue measure on $\mathbb{R}$, then $\liminf_{n \rightarrow \infty} \|\hat f_{h_{s,n}}-f\|_1
\geq ab > 0$ with probability $1$.
\end{enumerate}
\end{theorem}
\begin{proof}
1. It follows that, with probability $1$, 
\begin{align*}
c n^\varepsilon h_{s,n}^\varepsilon & \leq n^\varepsilon d_\infty(F,F_{h_{s,n}}) \\
& \leq n^\varepsilon \left(d_\infty(F,F_n) + d(F_n, \hat F_n^{h_{s,n}}) + d_\infty(F_{h_{s,n}}, \hat F_n^{h_{s,n}})\right)  \\
& = n^\varepsilon O((\log\log n/n)^{1/2}) + n^\varepsilon s(n) \\
& = o(1),
\end{align*}
and hence $nh_{s,n} = o(1)$. \\
2. If the support of $K$ is contained within
a compact interval $I$, then, since $\lambda\{ K_{h_{s,n}} \neq 0 \} \leq \lambda(I) $, we have almost surely
\begin{align*}
\lambda\{ \hat f_{h_{s,n}} \neq 0 \} 
\leq 2nh_{s,n}\lambda(I) = o(1)
\end{align*}
because of the first assertion. 
 It then follows almost surely  that
\begin{align*}
\liminf_{n \rightarrow \infty} \int | \hat f_{h_{s,n}}(x) - f(x) |dx
 \geq
\liminf_{n \rightarrow \infty} \int_{\{f \geq b\} \cap \{ \hat
f_{h_{s,n}} = 0 \}} f(x) dx  \geq ab.
\end{align*}\end{proof}

In the following example, we consider a family of densities with an infinite
peak and see that the using the discrepancy principle can lead to consistent or 
inconsistent estimates depending on the sharpness of the peak: 

\begin{example} \label{dpinkonsistenz}
Let \begin{align}
    K(x)=(3/4)(1-x^2)\mathbb{I}(|x|\leq 1) \label{epa} 
    \end{align}
 denote the Epanechnikov kernel and choose 
$s(n)$. 
Consider the distribution of $X:=U^\beta$ for $\beta \in [1,\infty)$, where $U$ is uniformly distributed on $[0,1]$.
With $\varepsilon=\beta^{-1}$ the density of $X$ is given by
\begin{align}
f(x):=\begin{cases}
\varepsilon x^{-(1-\varepsilon)} & 0 < x \leq 1 \cr
0 & \text{otherwise}
\end{cases}.\label{gegendichte}
\end{align}
The distribution function of $X$ is given by
\begin{align}
F(x):=\begin{cases}
0 & x \leq 0 \cr
x^{\varepsilon} & 0 < x \leq 1 \cr
1 & x >1
\end{cases}. \label{vfgegen}
\end{align}
It is easy to see that $F \in C^{0,\alpha}$ iff $\alpha \leq \varepsilon$. 

First consider the case that
$\sqrt{\frac{\log\log n}{n}}=o(s(n))$ and $n^\varepsilon s(n) \rightarrow \infty$. 
Then the conditions of Theorem 
\ref{konsistenz} are fulfilled and the estimator will be consistent w.r.t $L_1$-distance. 
Note that if $\sqrt{\frac{\log\log n}{n}}=o(s(n))$ then we trivially
have $n^\varepsilon s(n) \rightarrow \infty$ for all $\varepsilon > 1/2$.

Now consider the case that $0<\varepsilon <1/2$ and $n^\varepsilon s(n)=o(1)$.

By elementary integration, we obtain 
$$ |F(h) - (F\ast K_h)(h)| = \underbrace{\left(1 - \frac{3\cdot 2^{{\left(\varepsilon + 1\right)}}}{\varepsilon^{2} + 5 \varepsilon + 6} \right)}_{=:c>0}h^\varepsilon.$$
for $h<1$ and hence $$ d_\infty(F,F_h) \geq ch^\varepsilon $$
for $h<1=:h_0$. 
Since $K$ is compactly supported, inconsistency w.r.t. the $L_1$ distance directly follows from
the second assertion of Theorem \ref{dpinkonsistenzthm}. 
\end{example}

\section{Rates for the bandwidths}

 In the following, 
we consider threshold functions $s(n)$ that go to $0$ at different speeds: 
\begin{itemize}
\item $s(n) = o\left((\log\log n/n)^{1/2}\right)$ (Theorem \ref{dpnaiv}),
\item $s(n) \asymp (\log\log n/n)^{1/2}$ (Theorem \ref{SatzLIL}),
\item $(\log\log n/n)^{1/2} = o(s(n))$ (Theorem \ref{ELR}).
\end{itemize}

The versions of the discrepancy principle for kernel estimators previously proposed in the 
literature can be obtained by choosing a threshold function that belongs to one of these 
classes. 

To obtain more precise statements about the order of the chosen bandwidth, we need
 some additional assumptions of $f$ and $K$. In this section, we suppose that
$f$ is in a Sobolev space defined by  
$$ W^{\ell,1} := \{ f : f, f^{(1)}, \dots, f^{(\ell)}  \in  L_1(\mathbb{R}) \}.
$$ where $\ell \geq 2$ is the order of the Kernel $K$.

The following Lemma is a slight generalization of 
a similar result by \cite{ELRdisc,EggLaR1}, who only considered nonnegative symmetric kernels of 
order $\ell=2$ and $d=d_\infty$. The proof is left out since the first part is completely 
analogous to \citet[Lemma 6.15 a]{EggLaR1} and the second part is easy.  

\begin{lemma}\label{biasterm} 
Suppose that $f \in W^{\ell,1}(\mathbb{R})$ and $K$ is a kernel of order $\ell \geq 2$. Then we have: \begin{enumerate}
\item $F_h (x) - F(x) = \frac{(-1)^\ell}{\ell!} k_\ell f^{(\ell-1)}(x)h^\ell (1+o(1))$
uniformly in $x \in \mathbb{R}$.
\item $d(F_h,F) = \frac{1}{\ell!} k_\ell d(f^{(\ell-1)},0) h^\ell (1+o(1))$.
\end{enumerate}
\end{lemma}

The approximations given in Lemma \ref{biasterm} are only valid for sufficiently small $h$. 
Since by Theorem \ref{hgegen0} for $n \rightarrow \infty$ we have that $h_{s,n} \rightarrow 0$ almost surely with 
$h_{s,n}$ chosen by the discrepancy principle, terms of order $o(1)$ for $h \rightarrow 0$ are also
of order $o(1)$ for $n \rightarrow \infty$. 

The most simple and intuitive implementation of the discrepancy principle is based on a goodness-of-fit
test for a fixed level independent of $n$: the {\em Data Approximation} approach prescribes that for a given
data set, one should choose the simplest model that could have generated the data \citep{Dav08data}. 
For kernel density estimation, this results in the discrepancy principle (\ref{dpgleich}) with   
$d=d_\infty$ or $d=d_{kuip,k}$ and 
$s(n)=cn^{-1/2}$ with $c$ chosen as an appropriate quantile of $\sqrt{n} d(F_n,F)$.
Generally, the {\em Data Approximation} approach seems to suggest using extreme quantiles (95\%, 99\%). 
In Example 10 of \cite{Dav95data}, a discrepancy principle based on the 98\%-quantile of the 
Kuiper distance is used, which is then combined with a further criterion, the so-called extreme value feature.
In contrast, (for the Kolmogorov distance) \citet[Ch. 7.5.1]{VapNature} suggests to use the median 
 or even the mode, which is approximately located at $0.74$. Using 
$c=0.6$ is suggested in \cite{VapLearning}, \cite{Mar89density} suggests $c=0.7$ or $c=0.5$. 
With $c=0.6$ , the estimated distribution function is required to lie in 
a 14\% confidence band. However, $c$ has no effect on the rate with which $h_{s,n}$ 
converges to $0$:

\begin{theorem}\label{dpnaiv}
For $f \in W^{\ell,1}$, $K$ Kernel of order $\ell$ and $s(n) =
O\left(\sqrt{\frac{\log\log n}{n}}\right)$,
 $$ h_{s,n} = O(n^{-\frac{1}{2\ell}}(\log \log
n)^{\frac{1}{2\ell}})$$ almost surely.
\end{theorem}
\begin{proof}
According to Lemma \ref{biasterm}, we have a.s.
\begin{align*}
\frac{1}{\ell!} k_\ell \|f^{(\ell-1)}\|_\infty {h_{s,n}}^\ell(1 + o(1)) &= d_\infty(F_{h_{s,n}},F) \\
&\leq d_\infty(F, F_n) + d(F_n, \hat F_n^{h_{s,n}}) + d_\infty (\hat
F_n^{h_{s,n}}, F_{h_{s,n}}) \\
&= O\left((\log\log n/n)^{1/2}\right)+s(n) \\
&= O\left((\log\log n/n)^{1/2}\right).
\end{align*}
The second term in parentheses on the left-hand side is not only of order
$o(1)$ for ${h_{s,n}}
\rightarrow 0$,
but also $o(1)$ for $n \rightarrow \infty$ since, by Theorem 
\ref{hgegen0}, $n \rightarrow \infty$ almost surely implies ${h_{s,n}} \rightarrow 0$.
Solving for $h_{s,n}$ then proves the claim.
\end{proof}

Theorem \ref{dpnaiv} shows that for $f \in W^{\ell,1}$ and $K$ kernel of order $\ell$, 
an upper bound for the bandwidth (and hence the bandwidth itself) converges to $0$ 
at a faster rate than the optimal bandwidths according to most criteria, which behave like $h \asymp n^{-\frac{1}{2\ell+1}}$ (although this problem becomes less severe as
$\ell$ increases). 
The reason is that density estimation is an ill-posed problem that requires regularization. 
For sufficiently large $n$, the Kolmogorov-Smirnov-test with
fixed level will detect the difference between $F$ and $F_h$, even if $h$ is chosen optimally.
 This leads to a bandwidth that is too small. The incompatibility of 
optimal bandwidths with confidence sets based on the Kolmogorov-Smirnov or Kuiper tests has also been observed in \cite{Dav95data}, \cite{ELRdisc} 
and \cite{HjoWal2001kernel}.  
Asymptotically, the estimated distribution function is too close to the empirical distribution function, leading to undersmoothing. However, the simulations 
in Section 4 show that discrepancy principles based on extreme quantiles of goodness-of-fit tests still oversmooth even for sample sizes as large as
$n=2500$, while the version proposed by Vapnik ($c=0.6$) works quite well for the sample sizes considered. 

Theorem \ref{dpnaiv} is applicable to threshold functions of the form $s(n)=c\sqrt{\frac{\log \log n}{2n}}$, but more precise results are possible 
when $c$ is large enough. A threshold of this form is motivated by the law of the iterated logarithm for 
$d(F_n,F)$, and is in a sense the closest 
analogue to the upper bound on the error in deterministic inverse problems. \cite{AidVap89dens} considered the case where $c=(1+\tilde k+\varepsilon)$
for kernels $K$ of order $\ell$, $\tilde k = \|K\|_1$ and $d=d_\infty$.  The next theorem is a slight extension of 
their theorem \cite[Sec. 3]{AidVap89dens} 
that now additionally includes the case of $d=d_{kuip,k}$ and has essentially the same proof, see 
pp. 38 in \cite{Diss} for details.

\begin{theorem} \label{SatzLIL}
For $f \in W^{\ell,1}$, $K$ kernel of order $\ell \geq 2$, $\tilde k = \|K\|_1$ and $s(n) =
c_d(\tilde k+1+\varepsilon)\left(\log \log n/2n\right)^{1/2}$, we have with probability 1:  
\begin{align*}
\liminf_{n \rightarrow \infty} \frac{h}{(\log \log n/2n)^{1/2\ell}} &\geq
 \left(\frac{c_d \varepsilon \ell!}{k_\ell
d(f^{(\ell-1)},0)}\right)^{\frac{1}{\ell}} \\
\limsup_{n \rightarrow \infty} \frac{h}{(\log \log n/2n)^{1/2\ell}} &\leq
 \left(\frac{c_d (2\tilde k + 2 + \varepsilon ) \ell!}{k_\ell
d(f^{(\ell-1)},0)}\right)^{\frac{1}{\ell}}
\end{align*}

\end{theorem}

The theorem gives an upper and a lower bound on the selected bandwidth which are of the same order, and which again go to $0$ faster
than the optimal bandwidths according to most criteria. 

Exact results on the limiting behavior of the bandwidth chosen by the discrepancy principle can be 
obtained in the case where 
$s(n)$ converges to $0$ at a slower rate than $d(F_n,F)$. Noting that discrepancy principles based
on fixed quantiles or the law of the iterated logarithm lead to undersmoothing, 
\cite{ELRdisc,EggLaR1} introduce a rate-corrected version. 
For a symmetric, nonnegative kernel, they propose to choose $h$ as a solution of
 $$ d_\infty(F_n, \hat F_n^h) = 0.35 n^{-2/5}. $$ The choice of the exponent implies that the 
smoothing parameter goes to $0$ at the optimal rate. The next theorem is a generalization of the main
result in \cite{ELRdisc} and Chapter 7.6 of \cite{EggLaR1}. Our version is also applicable in the case of $d=d_{kuip,k}$ and allows
for  higher order kernels. 

\begin{theorem}\label{ELR}
For $f \in W^{\ell,1}$, $K$ kernel of order $\ell$ and $s(n) =
cn^{-\gamma}$ for $c>0$ and $0 <\gamma < 1/2$ we have almost surely:
\begin{align}
 h_{s,n} =  \left(\frac{c\ell!}{k_\ell
d(f^{(\ell-1)},0)}\right)^{\frac{1}{\ell}}
n^{-\frac{\gamma}{\ell}}\left(1 + o(1)\right). \label{ERbandbreite}
\end{align}
\end{theorem}
\begin{proof}
Using the triangle inequality, we have with probability $1$ that
\begin{align*}
| d(F_n, \hat F_n^{h_{s,n}}) - d(F,F_{h_{s,n}}) | \leq d(F_{h_{s,n}}, \hat
F_n^{h_{s,n}}) + d(F,F_n)=O \left( \sqrt{\frac{\log \log n}{n}} \right) .
\end{align*}
Combining this with Lemma \ref{biasterm} and again observing that, by Theorem \ref{hgegen0}, the $o(1)$ term for $h_{s,n}\rightarrow 0$
 is also of order $o(1)$ for $n \rightarrow \infty$, we have
\begin{align*}
\frac{1}{\ell!} k_\ell d(f^{(\ell-1)},0) {h_{s,n}}^\ell(1+o(1))
&= cn^{-\gamma} + O\left((\log\log n/n)^{1/2}\right) 
\end{align*}
which implies that 
\begin{align*}
{h_{s,n}}
&= \left(\frac{c\ell!}{k_\ell d(f^{(\ell-1)},0)}\right)^{\frac{1}{\ell}}
n^{-\frac{\gamma}{\ell}}\left(1 + o(1)\right).
\end{align*}
\end{proof}

The theorem implies that for a kernel of order $\ell$ and a threshold function of the form 
$s(n)=cn^{-\gamma}$ with $\gamma=\ell/(2\ell+1)$ the chosen bandwidth is -- 
for sufficiently smooth $f$ -- of the optimal order $h = \alpha n^{-1/(2\ell+1)}$ with respect to 
the $L_1$ or $L_2$ risks. The constant $\alpha$ depends on $c$ and the unknown true density $f$ 
and is not equal to the optimal one according to any of the standard criteria. Eggermont and LaRiccia
choose $c=0.35$ based on simulations. Noting that $ s(n)=cn^{-2/5} = (cn^{1/10})n^{-1/2}$ we can interpret the threshold function in 
terms of confidence levels that depend on $n$. For $c=0.35$, the confidence level is below $0.5$ up to $n=5624$.

In principle, constants suitable for other classes of densities, other distances or higher order kernels 
can also be chosen using simulations. But Theorem \ref{ELR} also allows for a different approach: Discrepancy 
principles that can be guaranteed to asymptotically choose the optimal bandwidths for a reference density.  
In the following example, we will sketch this approach for the normal distribution and the $L_2$-optimal bandwidth.
\begin{example}\label{l2ref}
The asymptotically $L_2$-optimal bandwidth for a kernel
of order $\ell$ is given by  
\begin{align}
 h_{opt}=\left( \frac{{\ell!}^2 \|K\|^2_2}{2\ell k_\ell^2\|f^{(\ell)}\|_2^2} \right)^{\frac{1}{2\ell+1}}n^{-\frac{1}{2\ell+1}}, \label{L2opth}
\end{align}
\citep[p.33]{WandJones}. Equating (\ref{ERbandbreite}) and (\ref{L2opth}) yields 
$\gamma=\frac{\ell}{2\ell+1}$ and 
\begin{align}
 c=\left(\frac{\|K\|^{2\ell}_2k_\ell}{(2\ell)^\ell \ell!}\right)^{\frac{1}{2\ell+1}}
\frac{d(f^{(\ell-1)},0)}{\|f^{(\ell)}\|_2^{2\ell/(2\ell+1)}}. \label{NRc}
\end{align}
The first factor only depends on the kernel and is invariant w.r.t. rescaling of the kernel.
The second factor only depends on the shape of $f$ and does not change when $f$ is translated or rescaled. 
For $f$ the standard normal density, we obtain $c=0.1357$ for the Gaussian and $c=0.1331$ for the 
Epanechnikov kernel (\ref{epa}) when using $d=d_\infty$. Both values are much smaller than $c=0.35$ as suggested by 
\cite{ELRdisc} independent of the kernel. This choice forces the estimate of the distribution function to 
lie extremely close to the empirical distribution function, causing severe undersmoothing even for very 
large sample sizes. For $d=d_{kuip,1}$, we obtain $c=0.2715$ for the Gaussian and $c=0.2661$ for the Epanechnikov kernel.
Similar calculations can be carried out for a bandwidth minimizing an upper bound for the $L_1$ risk, but 
these lead to even smaller values of $c$ \citep[pp. 44-45]{Diss}. 
\end{example}

\section{Simulation Study}

\begin{figure}
\begin{center}
\includegraphics[width=11cm,height=8.25cm]{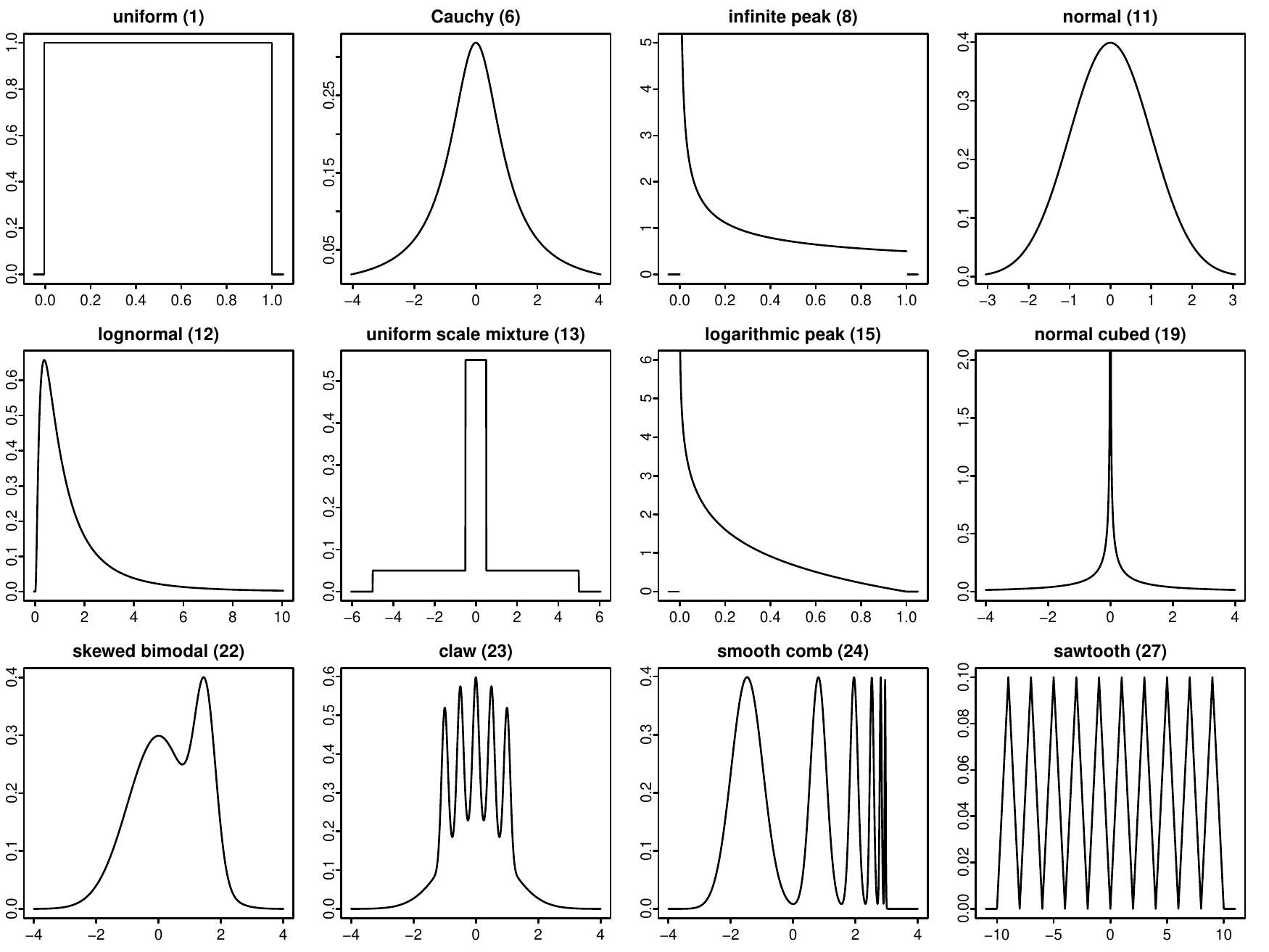}
\caption{The densities used.}\label{testbeds}
\end{center}
\end{figure}
In this section, we explore how well different versions of the discrepancy principle work in practice. Mainly in the 1980s and 1990s, several large
simulations studies on bandwidth choice methods for kernel density estimators have been conducted, of which we just mention \cite{CaoCueGon94study} 
(with a focus on the $L_2$-risk) and  \cite{BerDev} and \cite{Dev99universal} in an $L_1$-context. To the best of our knowledge, there is no larger
simulation study on kernel estimators that includes any version of the discrepancy principle, although in \cite{Dev99universal} 
the version proposed in \cite{ELRdisc} is mentioned but not included in the study. There are some smaller simulation 
studies to be found in the publications in which a particular version of the discrepancy principle is suggested or directly building on these 
\citep{Mar89density,ELRdisc,EggLaR1}. 
Our simulation study is a replication of a part of the more extensive study described in \cite{Diss}. 
The aim is not to find a 'best' method but to explore whether methods based on the discrepancy principle perform 
reasonably well at all. Since the discrepancy principle is not designed with any
specific risk in mind, we look at both the $L_1$- and $L_2$-risk (where applicable).

We use the Epanechnikov kernel as given in (\ref{epa}) and choose the bandwidth as a solution of  
$d(F_n,\hat F_n^h)=s(n)$. In \cite{ELRdisc}, a secant method is proposed for solving this equation, but we use the
related regula falsi which we found to be more stable. Occasionally, there may be multiple solutions but we ignore this and
take the first solution found. 

We compare the following versions of the discrepancy principle:
\begin{itemize}
 \item Two versions based on the $0.5$ and $0.95$ quantiles of the Kolmogorov-Smirnov statistic: $d=d_\infty$ and 
$s(n)=cn^{-1/2}$ with $c=0.83$ and $c=1.36$. These methods are denoted by {\bf KS .5} and {\bf KS .95}, respectively.
 \item The version proposed by Vapnik: $d=d_\infty$ and 
$s(n)=0.6n^{-1/2}$. Denoted by {\bf V}.
 \item The rate-corrected version proposed by Eggermont and LaRiccia: $d=d_\infty$ and $s(n)=0.35n^{-2/5}$. Denoted by {\bf E-LR}.
In contrast to the other versions considered here, this one uses a threshold function
for which the assumptions in Theorem \ref{konsistenz} are fulfilled.
 \item Two versions based on $0.5$ and $0.95$ quantiles of the Kuiper statistic: $d=d_{kuip,1}$ and 
$s(n)=cn^{-1/2}$ with $c=1.22$ and $c=1.75$. Denoted by {\bf Kuip .5} and {\bf Kuip .95}, respectively.
 \item The method based on a normal reference density as given in Example \ref{l2ref}: $d=d_\infty$ and $s(n)=0.1331n^{-2/5}$. Denoted by {\bf L2NR}. 
\end{itemize}

For comparison, we include $L_2$ cross-validation as described in \cite{CelRob08cv} 
(their Formula 13 with $p=1$). This is denoted by ${\bf L2CV}$. The more extensive simulations
in \cite{Diss} include several more variants of the discrepancy principle, a few more standard methods for comparison, and all of
the 28 densities from \cite{BerDev}. For the sake of brevity, here we just focus on a smaller subset but the conclusions 
are largely the same. 

We draw 250 samples of sizes 100, 1000 and 2500 from 12 of the 28 test bed densities introduced in \cite{BerDev}. 
For this, we use the {\tt R}-package {\tt benchden} \citep{benchden,benchdenpaper}. The set of densities is depicted in 
Figure \ref{testbeds}. We use the same numbering for the densities as in \cite{BerDev}.

\begin{figure}[t]
\begin{center}
\includegraphics[width=11cm,height=9cm]{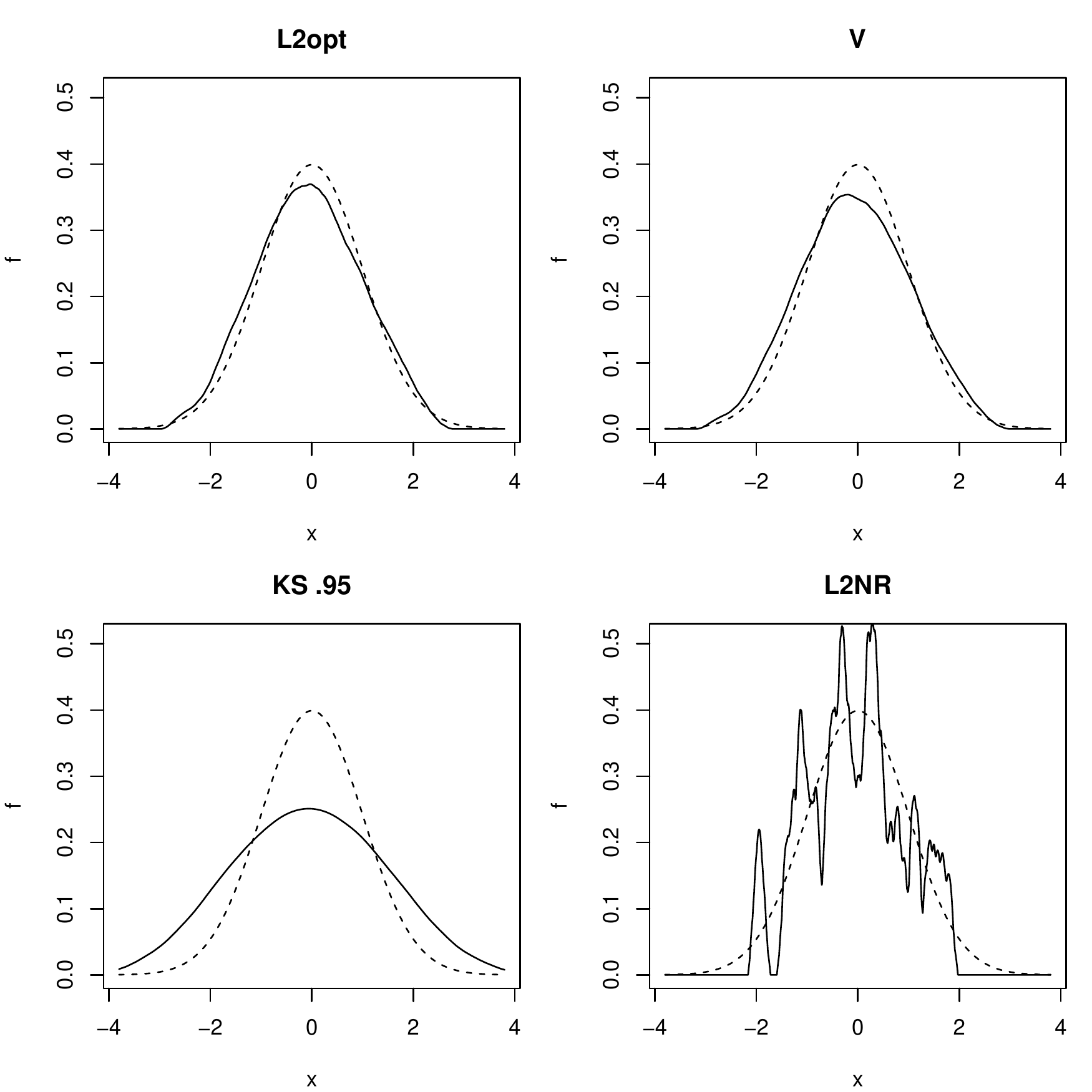}
\caption{Kernel estimates for a normal sample ($n=100$). The bandwidth used are the 
$L_2$-optimal choice (\ref{L2opth}) ({\bf L2opt}) and three variants of the discrepancy principle.}\label{bspnormbild}
\end{center}
\end{figure}

Figure \ref{bspnormbild} shows typical kernel estimates for a normal sample of size $100$.
In the first panel, the $L_2$-optimal bandwidth  (\ref{L2opth}) was chosen. 
The second panel shows the result obtained using {\bf V}, which gives 
a fairly good result. The bandwidth chosen using {{\bf KS .95}} is obviously too large, 
although it will be too small asymptotically. 
The bandwidth in the fourth panel has been chosen using 
{\bf L2NR}. Although this will asymptotically result in the 
optimal bandwidth, the estimate is severely undersmoothed. 

The estimated $L_1$ and (squared) $L_2$ risks and the arithmetic means of the chosen bandwidths for all densities and 
sample sizes considered here are 
given in Tables \ref{kerL1}, \ref{kerL2} and \ref{kerbw}, respectively. The smallest risk for each scenario has been 
highlighted. Note that Table \ref{kerL2} omits densities 8 and 19, 
since these are not in $L_2$.

\begin{table}[t]
\begin{center}
{\footnotesize\begin{tabular}{llllllllll}
Density & n & L2CV & E-LR & V & KS .5 & KS .95 & Kuip .5 & Kuip .95 & L2NR \\ 
   \hline
   \hline
1 & 100 & 0.2518 & \underline{\bf 0.2421 } & 0.2511 & 0.3018 & 0.4735 & 0.2884 & 0.3667 & 0.3871 \\ 
   & 1000 & 0.1043 & 0.1095 & \underline{\bf 0.1035 } & 0.1204 & 0.1831 & 0.1086 & 0.1341 & 0.1132 \\ 
   & 2500 & 0.0774 & 0.0775 & \underline{\bf 0.0734 } & 0.0803 & 0.1131 & 0.0744 & 0.085 & 0.0819 \\ 
   \hline
6 & 100 & 0.3316 & 0.3324 & \underline{\bf 0.3293 } & 0.3583 & 0.5169 & 0.3442 & 0.4163 & 0.6787 \\ 
   & 1000 & 0.1637 & 0.158 & 0.1581 & 0.1627 & 0.2097 & \underline{\bf 0.1576 } & 0.1714 & 0.23 \\ 
   & 2500 & 0.1242 & 0.1247 & 0.1248 & 0.1258 & 0.1493 & \underline{\bf 0.1237 } & 0.1294 & 0.1547 \\ 
   \hline
8 & 100 & 0.3784 & 0.3892 & 0.3741 & \underline{\bf 0.3521} & 0.4426 & 0.3608 & 0.4368 & 0.9115 \\ 
   & 1000 & 0.3584 & 0.2679 & 0.3021 & 0.2365 & \underline{\bf 0.1892 } & 0.2222 & 0.1907 & 0.6479 \\ 
   & 2500 & 0.3576 & 0.2361 & 0.2903 & 0.2214 & \underline{\bf 0.1608 } & 0.2047 & 0.1628 & 0.5845 \\ 
   \hline
11 & 100 & 0.1571 & \underline{\bf 0.1562 } & 0.1586 & 0.2082 & 0.4108 & 0.1886 & 0.2861 & 0.4559 \\ 
   & 1000 & 0.0613 & 0.0641 & \underline{\bf 0.0604 } & 0.0735 & 0.1326 & 0.0627 & 0.0859 & 0.1047 \\ 
   & 2500 & 0.0432 & 0.0472 & \underline{\bf 0.0428 } & 0.0503 & 0.0864 & 0.0438 & 0.0567 & 0.0625 \\ 
   \hline
12 & 100 & 0.2863 & 0.2758 & \underline{\bf 0.2744 } & 0.2937 & 0.4023 & 0.2986 & 0.3684 & 0.5874 \\ 
   & 1000 & 0.1313 & \underline{\bf 0.1249 } & 0.1251 & 0.1267 & 0.1471 & 0.1272 & 0.1416 & 0.167 \\ 
   & 2500 & 0.0959 & \underline{\bf 0.0897 } & 0.09 & 0.0903 & 0.1019 & 0.0907 & 0.0992 & 0.1083 \\ 
   \hline
13 & 100 & 0.3723 & \underline{\bf 0.3627} & 0.3647 & 0.4048 & 0.596 & 0.3922 & 0.4764 & 0.707 \\ 
   & 1000 & 0.1945 & 0.1816 & \underline{\bf 0.1808} & 0.1862 & 0.2251 & 0.1813 & 0.1936 & 0.2213 \\ 
   & 2500 & 0.151 & 0.1362 & 0.1363 & 0.1373 & 0.1583 & \underline{\bf 0.1356} & 0.1401 & 0.1627 \\ 
   \hline
15 & 100 & 0.3072 & 0.2946 & \underline{\bf 0.293} & 0.305 & 0.3876 & 0.3175 & 0.38 & 0.5593 \\ 
   & 1000 & 0.1668 & 0.1464 & 0.1534 & 0.1415 & 0.147 & \underline{\bf 0.1399} & 0.1482 & 0.2257 \\ 
   & 2500 & 0.1375 & 0.1114 & 0.1217 & 0.1088 & \underline{\bf 0.1036} & 0.1047 & 0.1041 & 0.1742 \\ 
   \hline
19 & 100 & 1.0525 & 1.0695 & 1.013 & 0.8101 & \underline{\bf 0.7173} & 0.8545 & 0.7274 & 1.6579 \\ 
   & 1000 & 1.4559 & 0.9915 & 1.107 & 0.8598 & \underline{\bf 0.5813} & 1.0131 & 0.7704 & 1.6414 \\ 
   & 2500 & 1.5706 & 0.9773 & 1.1507 & 0.9181 & \underline{\bf 0.5981} & 1.0842 & 0.8361 & 1.6392 \\ 
   \hline
22 & 100 & 0.1979 & \underline{\bf 0.1931} & 0.1985 & 0.2477 & 0.3981 & 0.2423 & 0.3232 & 0.4385 \\ 
   & 1000 & \underline{\bf 0.0804} & 0.0896 & 0.0832 & 0.1003 & 0.1531 & 0.0949 & 0.1299 & 0.1027 \\ 
   & 2500 & \underline{\bf 0.0566} & 0.0649 & 0.0577 & 0.0683 & 0.1015 & 0.0641 & 0.0856 & 0.0644 \\ 
   \hline
23 & 100 & 0.3829 & \underline{\bf 0.379} & 0.3826 & 0.4139 & 0.5278 & 0.4048 & 0.4592 & 0.4858 \\ 
   & 1000 & 0.3498 & 0.2826 & 0.2364 & 0.3301 & 0.3546 & 0.274 & 0.35 & \underline{\bf 0.1392} \\ 
   & 2500 & 0.2035 & 0.2158 & 0.1609 & 0.2371 & 0.3454 & 0.1813 & 0.27 & \underline{\bf 0.0922} \\ 
   \hline
24 & 100 & \underline{\bf 0.3775} & 0.4421 & 0.4743 & 0.6363 & 0.8741 & 0.5901 & 0.7567 & 0.4773 \\ 
   & 1000 & 0.173 & 0.2665 & 0.2369 & 0.3102 & 0.4793 & 0.2583 & 0.3487 & \underline{\bf 0.1695} \\ 
   & 2500 & 0.1274 & 0.2035 & 0.1689 & 0.2174 & 0.3394 & 0.1805 & 0.239 & \underline{\bf 0.1247} \\ 
   \hline
27 & 100 & 0.6238 & 0.583 & 0.5966 & 0.6489 & 0.7636 & 0.6376 & 0.6985 & \underline{\bf 0.4514} \\ 
   & 1000 & 0.5965 & 0.5366 & 0.4653 & 0.5488 & 0.5857 & 0.52 & 0.5337 & \underline{\bf 0.1682} \\ 
   & 2500 & 0.5339 & 0.4262 & 0.3127 & 0.4663 & 0.5398 & 0.3493 & 0.5147 & \underline{\bf 0.1275} \\ 
   \hline
\end{tabular}
}\caption{Results of the simulation study: Estimated $L_1$ risk for kernel estimators}
\label{kerL1}
\end{center}
\end{table}

\begin{table}[t]
\begin{center}
{\footnotesize\begin{tabular}{llllllllll}
Density & n & L2CV & E-LR & V & KS .5 & KS .95 & Kuip .5 & Kuip .95 & L2NR \\ 
   \hline
   \hline
1 & 100 & 0.0774 & \underline{\bf 0.0726} & 0.0755 & 0.0939 & 0.152 & 0.0889 & 0.117 & 0.2511 \\ 
   & 1000 & \underline{\bf 0.0205} & 0.0248 & 0.022 & 0.0295 & 0.0526 & 0.0245 & 0.0349 & 0.0215 \\ 
   & 2500 & \underline{\bf 0.012} & 0.0156 & 0.013 & 0.0169 & 0.0299 & 0.0138 & 0.019 & 0.0122 \\ 
   \hline
6 & 100 & \underline{\bf 0.008} & \underline{\bf 0.008} & 0.0082 & 0.0129 & 0.0336 & 0.0111 & 0.0203 & 0.0529 \\ 
   & 1000 & \underline{\bf 0.0014} & 0.0016 & \underline{\bf 0.0014} & 0.0021 & 0.0055 & 0.0016 & 0.0027 & 0.0028 \\ 
   & 2500 & \underline{\bf 6e-04} & 8e-04 & \underline{\bf 6e-04} & 9e-04 & 0.0024 & 7e-04 & 0.0012 & 0.001 \\ 
   \hline
11 & 100 & 0.0072 & 0.007 & \underline{\bf 0.0069} & 0.0104 & 0.0332 & 0.0088 & 0.0177 & 0.076 \\ 
   & 1000 & 0.0011 & 0.0011 & \underline{\bf 0.001} & 0.0014 & 0.004 & 0.0011 & 0.0018 & 0.0036 \\ 
   & 2500 & 6e-04 & 6e-04 & \underline{\bf 5e-04} & 7e-04 & 0.0017 & 6e-04 & 8e-04 & 0.0012 \\ 
   \hline
12 & 100 & 0.0235 & \underline{\bf 0.0213} & 0.0221 & 0.0309 & 0.0639 & 0.0325 & 0.0541 & 0.1137 \\ 
   & 1000 & \underline{\bf 0.0043} & 0.0049 & 0.0044 & 0.0059 & 0.0118 & 0.006 & 0.0103 & 0.0063 \\ 
   & 2500 & \underline{\bf 0.0022} & 0.0027 & \underline{\bf 0.0022} & 0.0029 & 0.0058 & 0.003 & 0.0052 & 0.0025 \\ 
   \hline
13 & 100 & 0.039 & \underline{\bf 0.0379} & 0.0394 & 0.0511 & 0.1017 & 0.048 & 0.0686 & 0.1124 \\ 
   & 1000 & \underline{\bf 0.0124} & 0.015 & 0.0137 & 0.017 & 0.0264 & 0.0148 & 0.0192 & 0.0126 \\ 
   & 2500 & 0.0077 & 0.0103 & 0.0088 & 0.011 & 0.0169 & 0.0094 & 0.0122 & \underline{\bf 0.0076} \\ 
   \hline
15 & 100 & 0.3402 & \underline{\bf 0.2946} & 0.3034 & 0.3663 & 0.5308 & 0.397 & 0.5167 & 0.7157 \\ 
   & 1000 & 0.1059 & 0.1104 & \underline{\bf 0.1043} & 0.122 & 0.181 & 0.1338 & 0.1836 & 0.1262 \\ 
   & 2500 & 0.0688 & 0.075 & \underline{\bf 0.0682} & 0.0784 & 0.1128 & 0.0862 & 0.1168 & 0.0788 \\ 
   \hline
22 & 100 & 0.0124 & \underline{\bf 0.0119} & 0.0125 & 0.0181 & 0.0337 & 0.0176 & 0.0263 & 0.0725 \\ 
   & 1000 & \underline{\bf 0.0021} & 0.0028 & 0.0023 & 0.0036 & 0.0081 & 0.0032 & 0.006 & 0.0034 \\ 
   & 2500 & \underline{\bf 0.001} & 0.0015 & 0.0011 & 0.0017 & 0.0039 & 0.0014 & 0.0027 & 0.0013 \\ 
   \hline
23 & 100 & \underline{\bf 0.0536} & 0.0546 & 0.0544 & 0.0572 & 0.0768 & 0.0561 & 0.0639 & 0.1051 \\ 
   & 1000 & 0.0485 & 0.034 & 0.0238 & 0.0457 & 0.0472 & 0.0319 & 0.0501 & \underline{\bf 0.0074} \\ 
   & 2500 & 0.0236 & 0.0202 & 0.011 & 0.0244 & 0.05 & 0.014 & 0.0317 & \underline{\bf 0.0033} \\ 
   \hline
24 & 100 & \underline{\bf 0.0444} & 0.0547 & 0.0597 & 0.0856 & 0.127 & 0.0777 & 0.1073 & 0.0813 \\ 
   & 1000 & \underline{\bf 0.0116} & 0.028 & 0.0239 & 0.034 & 0.0578 & 0.0269 & 0.0392 & 0.0117 \\ 
   & 2500 & \underline{\bf 0.006} & 0.0202 & 0.0153 & 0.022 & 0.0378 & 0.017 & 0.0248 & 0.0075 \\ 
   \hline
27 & 100 & 0.0205 & 0.0198 & 0.02 & 0.0211 & 0.0239 & 0.0208 & 0.0222 & \underline{\bf 0.0185} \\ 
   & 1000 & 0.0192 & 0.0188 & 0.0146 & 0.0191 & 0.0192 & 0.0178 & 0.0176 & \underline{\bf 0.0023} \\ 
   & 2500 & 0.0171 & 0.0123 & 0.007 & 0.0146 & 0.0181 & 0.0085 & 0.0174 & \underline{\bf 0.0014} \\ 
\hline
\end{tabular}
}\caption{Results of the simulation study: Estimated squared $L_2$ risk for kernel estimators}
\label{kerL2}
\end{center}
\end{table}

\begin{table}[t]
\begin{center}
{\footnotesize\begin{tabular}{llllllllll}
Density & n & L2CV & E-LR & V & KS .5 & KS .95 & Kuip .5 & Kuip .95 & L2NR \\ 
   \hline
   \hline
1 & 100 & 0.2267 & 0.2206 & 0.2443 & 0.3609 & 0.6297 & 0.3342 & 0.4784 & 0.0292 \\ 
   & 1000 & 0.0651 & 0.1073 & 0.0913 & 0.1286 & 0.2149 & 0.1055 & 0.1502 & 0.0363 \\ 
   & 2500 & 0.0381 & 0.077 & 0.0599 & 0.0838 & 0.139 & 0.0663 & 0.0946 & 0.0277 \\ 
   \hline
6 & 100 & 1.1911 & 1.1651 & 1.2968 & 1.9599 & 3.6053 & 1.7881 & 2.6064 & 0.148 \\ 
   & 1000 & 0.6465 & 0.8374 & 0.725 & 0.9776 & 1.4834 & 0.8259 & 1.1175 & 0.2384 \\ 
   & 2500 & 0.5264 & 0.7162 & 0.5793 & 0.7652 & 1.1179 & 0.644 & 0.8551 & 0.2483 \\ 
   \hline
8 & 100 & 0.0862 & 0.0439 & 0.0511 & 0.0946 & 0.2425 & 0.1193 & 0.2323 & 0.0061 \\ 
   & 1000 & 0.0036 & 0.0062 & 0.0046 & 0.0087 & 0.0232 & 0.0108 & 0.0223 & 0.001 \\ 
   & 2500 & 0.0013 & 0.003 & 0.0018 & 0.0035 & 0.0092 & 0.0043 & 0.0088 & 4e-04 \\ 
   \hline
11 & 100 & 0.9621 & 0.9097 & 1.0118 & 1.4927 & 2.4632 & 1.3745 & 1.9095 & 0.0977 \\ 
   & 1000 & 0.5867 & 0.711 & 0.6137 & 0.8288 & 1.2229 & 0.6992 & 0.9389 & 0.1732 \\ 
   & 2500 & 0.4932 & 0.6253 & 0.5052 & 0.6677 & 0.9607 & 0.5581 & 0.7385 & 0.1942 \\ 
   \hline
12 & 100 & 0.4389 & 0.4423 & 0.4865 & 0.6981 & 1.2262 & 0.7238 & 1.0752 & 0.0654 \\ 
   & 1000 & 0.1945 & 0.2675 & 0.2374 & 0.3063 & 0.4542 & 0.3081 & 0.4212 & 0.0996 \\ 
   & 2500 & 0.1456 & 0.2163 & 0.1815 & 0.2293 & 0.3285 & 0.2326 & 0.3108 & 0.0994 \\ 
   \hline
13 & 100 & 0.4289 & 0.4419 & 0.4889 & 0.7277 & 1.3762 & 0.6792 & 0.9784 & 0.0684 \\ 
   & 1000 & 0.1219 & 0.2108 & 0.1791 & 0.2533 & 0.4253 & 0.2072 & 0.2962 & 0.0716 \\ 
   & 2500 & 0.0723 & 0.1507 & 0.1167 & 0.1641 & 0.2737 & 0.1308 & 0.1871 & 0.0538 \\ 
   \hline
15 & 100 & 0.0903 & 0.0631 & 0.0703 & 0.1104 & 0.2223 & 0.1286 & 0.2116 & 0.0113 \\ 
   & 1000 & 0.0147 & 0.0208 & 0.0173 & 0.0257 & 0.0477 & 0.0304 & 0.0486 & 0.0065 \\ 
   & 2500 & 0.0077 & 0.0134 & 0.01 & 0.0148 & 0.027 & 0.018 & 0.0284 & 0.0043 \\ 
   \hline
19 & 100 & 0.0505 & 0.0327 & 0.0423 & 0.1231 & 0.6511 & 0.0916 & 0.263 & 0.0018 \\ 
   & 1000 & 4e-04 & 0.0029 & 0.0018 & 0.0051 & 0.0235 & 0.0026 & 0.0077 & 1e-04 \\ 
   & 2500 & 1e-04 & 0.0011 & 5e-04 & 0.0014 & 0.0066 & 7e-04 & 0.002 & 1e-04 \\ 
   \hline
22 & 100 & 0.8416 & 0.8513 & 0.9462 & 1.3895 & 2.3862 & 1.3423 & 1.937 & 0.1067 \\ 
   & 1000 & 0.4001 & 0.5509 & 0.4857 & 0.6315 & 0.9311 & 0.5928 & 0.8078 & 0.1784 \\ 
   & 2500 & 0.3116 & 0.4587 & 0.3814 & 0.4865 & 0.6895 & 0.4525 & 0.6014 & 0.1844 \\ 
   \hline
23 & 100 & 0.7869 & 0.6109 & 0.7071 & 1.1304 & 1.9354 & 1.0402 & 1.4977 & 0.0718 \\ 
   & 1000 & 0.5907 & 0.3161 & 0.2601 & 0.4032 & 0.7562 & 0.3033 & 0.4961 & 0.0945 \\ 
   & 2500 & 0.2641 & 0.2485 & 0.1956 & 0.2701 & 0.4789 & 0.2151 & 0.3065 & 0.0917 \\ 
   \hline
24 & 100 & 0.4771 & 0.7096 & 0.8014 & 1.2902 & 2.7506 & 1.1366 & 1.815 & 0.097 \\ 
   & 1000 & 0.1068 & 0.4089 & 0.3443 & 0.4997 & 0.8546 & 0.3915 & 0.5795 & 0.1214 \\ 
   & 2500 & 0.0604 & 0.3038 & 0.2328 & 0.3314 & 0.5653 & 0.257 & 0.3731 & 0.1004 \\ 
   \hline
27 & 100 & 5.7691 & 3.876 & 4.3972 & 6.8032 & 12.1572 & 6.2664 & 9.1695 & 0.4075 \\ 
   & 1000 & 4.7022 & 1.6839 & 1.3201 & 2.1725 & 3.9145 & 1.5474 & 2.6182 & 0.4851 \\ 
   & 2500 & 2.9109 & 1.2055 & 0.938 & 1.3235 & 2.405 & 1.0177 & 1.5117 & 0.4597 \\ 
   \hline
\end{tabular}
}\caption{Results of the simulation study: Chosen Bandwidth}
\label{kerbw}
\end{center}
\end{table}

In most cases, either {\bf L2CV} or one of {\bf V} and {\bf E-LR}, 
which perform very similarly, is the best method with respect to the $L_2$-risk. 
Although {\bf L2CV} usually selects smaller bandwidths than {\bf V} and {\bf E-LR}, 
the resulting risks are close in most cases. The methods based on quantiles of the Kolmogorov or 
Kuiper statistics ({\bf KS .5}, {\bf KS .95}, {\bf Kuip .5} and  {\bf Kuip .95}) choose larger bandwidths, which results in oversmoothing in most cases (although, by Theorem \ref{dpnaiv},
 these methods asymptotically suffer from undersmoothing). The methods based on the Kuiper statistic are usually better than
those based on the corresponding quantiles of the Kolmogorov-Smirnov statistic. The rather large amount of smoothing chosen
by these methods is beneficial w.r.t. $L_1$ risk for the Cauchy density (number 6), which is mainly due to the fact
that the $L_1$ loss penalizes errors in the tails quite heavily. 

The method {\bf L2NR}
chooses bandwidths that are much smaller than those chosen by the other methods. Although it is guaranteed to 
asymptotically choose the $L_2$ optimal bandwidth for the normal density, 
the results for both $L_1$ and $L_2$ risk are poor even when the true density is the normal (number 11). The small
bandwidths seem to be helpful for capturing the fine structure of multimodal densities 23, 27 and (less pronounced)
24. 

Except for 8, 15 and 19 all densities are bounded and hence fulfill the assumptions
of Theorem \ref{konsistenz}, such that the estimate based on a bandwidth chosen 
using {\bf E-LR} will be almost surely consistent w.r.t. the $L_1$-distance. 

Density 8 is the density of $U^2$, where $U$ is a 
uniform random variable on $[0,1]$. This corresponds to the density in Example
\ref{dpinkonsistenz} for $\varepsilon=1/2$, such that using {\bf E-LR} to select the bandwidth 
will result in a consistent estimate w.r.t. $L_1$-loss. The density is not in $L_2$.
With respect to $L_1$-risk, at least for larger sample sizes all versions of the discrepancy
principle except {\bf L2NR} perform better than {\bf L2CV}.  
The distribution function corresponding to density 15 is in $C^{0,\alpha}$ for any
$\alpha < 1$, and hence using {\bf E-LR} to select the bandwidth will also lead to $L_1$-consistent
estimates by Theorem \ref{konsistenz} and Corollary \ref{konscorr}. This density is in $L_2$.
Again, {\bf L2CV} performs worse than all variants of the discrepancy principle except
{\bf L2NR}.
Density number 19 is the density of $N^3$, where $N$ is a standard normal random variable. 
It is not in $L_2$ and it can be shown that 
$$ |F(h)-F\ast K_h (h)| = \frac{1}{\sqrt{2\pi}}h^{\frac{1}{3}}(1+o(1)),$$
where $F$ is the distribution function and $K_h$ the Epanechnikov-kernel with bandwidth $h$.
Hence, for $h$ small enough, the conditions of Theorem \ref{dpinkonsistenzthm}
are fulfilled with $\varepsilon=1/3$ and any $0< c < (\sqrt{2\pi})^{-1}$. From this it follows that
every version of the discrepancy principle considered in the simulation study 
will lead to inconsistent estimates w.r.t. $L_1$-loss almost surely.     
The main Theorem in \cite{Dev89nonconsistency}, 
which states that selecting the bandwidth 
using {\bf L2CV} will lead to $L_1$-inconsistent estimators for any sufficiently sharply peaked density  is not applicable to density 19. However, Table \ref{kerL1} shows
that {\bf L2CV} performs even worse than most versions of the discrepancy principle
in our simulations (with the versions choosing larger bandwidths doing relatively better). 

Overall, if the discrepancy principle is to be used for choosing a bandwidths,
from the simulations it seems that {\bf V} and {\bf E-LR} would be the 
versions of choice. Although they perform similarly in the simulation study, 
there are good theoretical reasons for preferring {\bf E-LR} as 
consistency can be guaranteed can be guaranteed for a large class of densities.
Generally, the simulations show that the asymptotic results are of limited use 
for the sample sizes considered here (even for $n=2500$!). This is not so much of 
a surprise since the asymptotic analysis is largely based on the law of the iterated logarithm.

\section{Conclusions}

The discrepancy principle is a fast and simple method of parameter choice that is also easy to implement. Although it 
is very popular in other branches of applied mathematics (namely in ill-posed problems theory), it has only rarely 
been used in density estimation. While there are many shortcomings -- it is not optimal in any sense and it can even 
lead to inconsistent estimates for some densities with infinite peaks --, 
some variants do work surprisingly well 
for a large set of different densities in simulations and consistency can -- at least for some
versions -- be guaranteed for a large class of densities including all square-integrable ones. 
The simulations also show that the behavior of methods based 
on the discrepancy principle may be quite different from the asymptotic behavior even for sample sizes as 
large as $n=2500$. Generally, asymptotic results do not help much in choosing the threshold function $s(n)$ -- 
the most striking example being the $L_2$ normal reference version {\bf L2NR} which is guaranteed to asymptotically
choose the $L_2$ optimal bandwidth for the normal distribution but performs very poorly even when the true density is the normal.    
Also taking into account the inconsistency for some densities (a problem that is actually shared by many popular 
bandwidth selectors), the method cannot be recommended in general.

\section*{Acknowledgement}

Large parts of the present work were part of the author's Ph.D. thesis \citep{Diss} at the Faculty of Statistics of the TU Dortmund University and has been supported
by the Collaborative Research Center {\em Statistical modeling of nonlinear
dynamic processes} (SFB 823, Project C1) of the German Research Foundation (DFG).
 The author wants to thank his supervisor Ursula Gather 
for her constant support and for fruitful discussions.

\end{document}